\documentclass[11pt]{article}

\usepackage{amsmath, amssymb, amsthm} 

\usepackage{dsfont}

\usepackage{tikz}

\allowdisplaybreaks
\expandafter\let\expandafter\oldproof\csname\string\proof\endcsname
\let\oldendproof\endproof
\renewenvironment{proof}[1][\proofname]{%
	\oldproof[\bf #1]%
}{\oldendproof}

\allowdisplaybreaks

\theoremstyle{plain}
\newtheorem{theorem}{Theorem}[section]
\newtheorem{lemma}[theorem]{Lemma}
\newtheorem{claim}[theorem]{Claim}
\newtheorem{proposition}[theorem]{Proposition}

\newtheorem{corollary}[theorem]{Corollary}

\newtheorem{problem}[theorem]{Problem}

\theoremstyle{definition}
\newtheorem{remark}[theorem]{Remark}
\newtheorem{definition}[theorem]{Definition}

\newcommand{\Bin}{\ensuremath{\textrm{Bin}}}

\newcommand{\eps}{\varepsilon}

\RequirePackage[normalem]{ulem} 
\RequirePackage{color}\definecolor{RED}{rgb}{1,0,0}\definecolor{BLUE}{rgb}{0,0,1} 

\newcommand{\ex}{\text{ex}}

\title{\vspace{-0.9cm}Induced subgraphs of $K_r$-free graphs and the Erd\H os--Rogers problem}

\author{
Lior Gishboliner\thanks{Department of Mathematics, University of Toronto, Canada. Email: lior.gishboliner@utoronto.ca}
\and 
Oliver Janzer\thanks{Department of Pure Mathematics and Mathematical Statistics, University of Cambridge, United Kingdom. Research supported by a fellowship at Trinity College. Email: oj224@cam.ac.uk}
\and 
Benny Sudakov\thanks{Department of Mathematics, ETH Z\"urich, Switzerland. Email: benjamin.sudakov@math.ethz.ch. Research supported in part by SNSF grant 200021-228014.}
}

\date{}

\begin{document}

\maketitle

\begin{abstract}
    For two graphs $F,H$ and a positive integer $n$, the function $f_{F,H}(n)$ denotes the largest $m$ such that every $H$-free graph on $n$ vertices contains an $F$-free induced subgraph on $m$ vertices. This function has been extensively studied in the last 60 years when $F$ and $H$ are cliques and became known as the Erd\H os--Rogers function. Recently, Balogh, Chen and Luo, and Mubayi and Verstra\"ete initiated the systematic study of this function in the case where $F$ is a general graph.
    
    Answering, in a strong form, a question of Mubayi and Verstra\"ete, we prove that for every positive integer $r$ and every $K_{r-1}$-free graph $F$, there exists some $\eps_F>0$ such that $f_{F,K_r}(n)=O(n^{1/2-\eps_F})$. This result is tight in two ways. Firstly, it is no longer true if $F$ contains $K_{r-1}$ as a subgraph. Secondly, we show that for all $r\geq 4$ and $\eps>0$, there exists a $K_{r-1}$-free graph $F$ for which $f_{F,K_r}(n)=\Omega(n^{1/2-\eps})$. Along the way of proving this, we show in particular that for every graph $F$ with minimum degree $t$, we have $f_{F,K_4}(n)=\Omega(n^{1/2-6/\sqrt{t}})$. This answers (in a strong form) another question of Mubayi and Verstra\"ete. Finally, we prove that there exist absolute constants $0<c<C$ such that for each $r\geq 4$, if $F$ is a bipartite graph with sufficiently large minimum degree, then $\Omega(n^{\frac{c}{\log r}})\leq f_{F,K_r}(n)\leq O(n^{\frac{C}{\log r}})$. This shows that for graphs $F$ with large minimum degree, the behaviour of $f_{F,K_r}(n)$ is drastically different from that of the corresponding off-diagonal Ramsey number $f_{K_2,K_r}(n)$.
\end{abstract}

\section{Introduction}

The Ramsey number $R(r,t)$ is the smallest $n$ such that every $n$-vertex graph contains a clique of size~$r$ or an independent set of size $t$. The study of this function is one of the most important problems in discrete mathematics. The instances that have  received the most attention are the ``diagonal case'' concerning $r=t$, and the case where $r$ is fixed and $t\rightarrow \infty$ (which is often called the ``off-diagonal case''). In this paper we will focus on the latter.

The first bound on this function was obtained by Erd\H os and Szekeres \cite{ESz35} in 1935, who proved that $R(r,t)=O(t^{r-1})$ for any fixed $r$ and $t\rightarrow \infty$.
Despite extensive research on the topic, the only (non-trivial) off-diagonal Ramsey number whose order of magnitude is known is $R(3,t)$. It was shown by Kim \cite{Kim95} in 1995 that $R(3,t)=\Omega(t^2/\log t)$, which matches an earlier upper bound by Ajtai, Koml\'os and Szemer\'edi \cite{AKSz80}.
Recently, a major breakthrough was obtained by Mattheus and Verstra\"ete \cite{Mattheus_Verstraete}, who proved that $R(4,t)\geq \Omega(t^{3}/(\log t)^4)$, matching the best known upper bound up to a polylogarithmic factor. Nevertheless, the problem of estimating $R(r,t)$ remains wide open for all $r\geq 5$, with the best bounds being $$c_1(r)\frac{t^{\frac{r+1}{2}}}{(\log t)^{\frac{r+1}{2}-\frac{1}{r-2}}}\leq R(r,t)\leq c_2(r) \frac{t^{r-1}}{(\log t)^{r-2}},$$
due to Bohman and Keevash \cite{BK10} and Ajtai, Koml\'os and Szemer\'edi \cite{AKSz80}, respectively.

In 1962, Erd\H os and Rogers \cite{ER62} considered the following generalization of the off-diagonal Ramsey problem. For positive integers $2\leq s<r$ and $n$, let $f_{s,r}(n)$ denote the largest $m$ such that every $K_r$-free graph on $n$ vertices contains a $K_s$-free induced subgraph on $m$ vertices. Note that the Ramsey problem is recovered as the special case $s=2$. The function $f_{s,r}$ has since become known as the Erd\H os--Rogers function and has attracted an extensive amount of research over the last 60 years (see, e.g. \cite{ER62,BH91,Kri94,Kri95,AK97,Sudakov,Sudakov_DRC,DR11,Dudek_Rodl,Wolfovitz,DM14,Dudek_Retter_Rodl,Gowers_Janzer,Janzer_Sudakov,MV_improved_bound}). 

In the last decade or so, there has been major progress towards finding the value of $f_{r-1,r}(n)$. Building on earlier work of Dudek and R\"odl \cite{DR11}, Wolfovitz \cite{Wolfovitz} proved that $f_{3,4}(n)\leq n^{1/2+o(1)}$, matching the easy lower bound $f_{3,4}(n)\geq n^{1/2}$ up to the $o(1)$ term. Later, it was shown by Dudek, Retter and R\"odl \cite{Dudek_Retter_Rodl} that for all $r\geq 4$, we have $f_{r-1,r}(n)\leq n^{1/2+o(1)}$, which is again tight up to the $o(1)$ term. Very recently, the upper bound was improved by Mubayi and Verstra\"ete \cite{MV_improved_bound}, who showed that $f_{r-1,r}(n)= O(n^{1/2}\log n)$, coming close to the best known lower bound $f_{r-1,r}(n)=\Omega\left(\frac{n^{1/2}(\log n)^{1/2}}{(\log \log n)^{1/2}}\right)$, observed in \cite{DM14}.

While the $s=r-1$ case is more or less settled, the next case $s=r-2$ is already open in general. For $r=4$, this problem is equivalent to determining the Ramsey numbers $R(4,k)$, and it follows from the recent breakthrough of Mattheus and Verstra\"ete \cite{Mattheus_Verstraete} that $f_{2,4}(n)\leq n^{1/3+o(1)}$, which is tight. Janzer and Sudakov \cite{Janzer_Sudakov} generalized this upper bound by proving that $f_{r-2,r}(n)\leq n^{\frac{1}{2}-\frac{1}{8r-26}+o(1)}$ holds for all $r\geq 4$. It is unknown whether this is tight for $r>4$; the best lower bound is $f_{r-2,r}(n)\geq n^{\frac{1}{2}-\frac{1}{6r-18}+o(1)}$, due to Sudakov \cite{Sudakov_DRC}.

Recently, Balogh, Chen and Luo \cite{Balogh_Chen_Luo} and Mubayi and Verstra\"ete \cite{MV_general_graphs} initiated a systematic study of the following generalization of the classical Erd\H os--Rogers function (see also \cite{HL24} for an earlier paper in this direction). For graphs $F$ and $H$ and a positive integer $n$, we write $f_{F,H}(n)$ for the largest $m$ such that every $H$-free graph on $n$ vertices contains an $F$-free induced subgraph on $m$ vertices. (Here $H$-free and $F$-free mean that they do not contain $H$ or $F$ as a not necessarily induced subgraph.) Both \cite{Balogh_Chen_Luo} and \cite{MV_general_graphs} are in fact mainly concerned with the case where $H$ is a clique, and this will be the focus of our paper as well. Note that this problem still closely resembles the original Ramsey problem; the only difference is that we are looking for a large $F$-free induced subgraph rather than a large independent set.

Among other results, Mubayi and Verstra\"ete \cite {MV_general_graphs} proved that for every (non-empty) triangle-free graph $F$, we have $f_{F,K_3}(n)=n^{1/2+o(1)}$, thereby resolving the case where $H$ is the triangle. Regarding the next case, namely that of $H=K_4$, they posed the following problem.

\begin{problem}[Mubayi--Verstra\"ete \cite{MV_general_graphs}] \label{problem:trianglefree vs K4}
    Is it true that for every triangle-free graph $F$, there exists $\eps=\eps_F>0$ such that $f_{F,K_4}(n)=O(n^{1/2-\eps})$?
\end{problem}

Our first result is an affirmative answer to Problem \ref{problem:trianglefree vs K4} in a more general form.

\begin{theorem} \label{thm:cliquefree vs clique}
    For every $r\geq 4$ and every $K_{r-1}$-free graph $F$, there exists $\eps=\eps_F>0$ such that $f_{F,K_r}(n)=O(n^{1/2-\eps})$.
\end{theorem}

The assumption that $F$ is $K_{r-1}$-free is necessary since if $F$ contains $K_{r-1}$, then $f_{F,K_r}(n)\geq f_{K_{r-1},K_r}(n)\geq n^{1/2+o(1)}$ for any $r$. Mubayi and Verstra\"ete also conjectured that the $1/2$ in the exponent in Problem \ref{problem:trianglefree vs K4} cannot be replaced by anything smaller, and that this is witnessed by taking $F=K_{t,t}$ for large enough $t$.

\begin{problem}[Mubayi--Verstra\"ete \cite{MV_general_graphs}] \label{problem:Ktt vs K4}
    Prove (or disprove) that for each $\eps>0$ there exists $t$ such that $f_{K_{t,t},K_4}(n)=\Omega(n^{1/2-\eps})$.
\end{problem}

We prove that this is indeed the case in the following more general form.

\begin{theorem} \label{thm:Ktt vs K4}
    For every $t$ and every graph $F$ with minimum degree $t$, $f_{F,K_4}(n)=\Omega(n^{1/2-6/\sqrt{t}})$.
\end{theorem}

In fact, using the same method, we prove that our Theorem \ref{thm:cliquefree vs clique} is tight for all $r$.

\begin{theorem} \label{thm:F vs Kr}
    For every $r\geq 4$ and $\eps>0$ there is a $K_{r-1}$-free graph $F$ such that $f_{F,K_r}(n)=\Omega(n^{1/2-\eps})$.
\end{theorem}

As a corollary of Theorem \ref{thm:Ktt vs K4}, we obtain the following result about graphs with large Tur\'an number.

\begin{corollary} \label{cor:large Turan}
    For every $\eps>0$ there exists some $\delta>0$ such that if a bipartite graph $F$ satisfies $\ex(m,F)=\Omega(m^{2-\delta})$, then $f_{F,K_4}(n)=\Omega(n^{1/2-\eps})$.
\end{corollary}

This shows that for bipartite graphs $F$ with large Tur\'an number, the exponent in $f_{F,K_4}(n)$ is close to $1/2$. This complements a result of Balogh, Chen and Luo \cite{Balogh_Chen_Luo} which states that if $\ex(m,F)=O(m^{1+\alpha})$ for some $\alpha\in [0,1/2)$, then $f_{F,K_4}(n)\leq n^{\frac{1}{3-2\alpha}+o(1)}$.

Motivated by Theorem \ref{thm:Ktt vs K4}, 
it is natural to ask what happens if instead of $K_4$ one considers the case of a general clique $K_r$. Our methods allow us to address this question as well, and we obtain the following rather accurate estimates on $f_{F,K_r}(n)$ when $F$ is a bipartite graph with large minimum degree.

\begin{theorem} \label{thm:Ktt vs Kr lower bound}
    For each $r\geq 4$ and $\varepsilon > 0$ there is $t_0$ such that for every $t \geq t_0$ and every graph $F$ with minimum degree $t$, we have
    $$
    f_{F,K_r}(n) = \Omega(n^{\frac{1}{\lceil \log_2 r \rceil} - \varepsilon}).
    $$
\end{theorem}
\begin{theorem}\label{thm:Ktt vs Kr upper bound}
      There is an absolute constant $C > 0$ such that for every $r$ and every bipartite graph~$F$, we have
    $$f_{F,K_r}(n)= O(n^{\frac{C}{\log r}}).$$
\end{theorem}

From the above two theorems, we see that there are absolute constants $c,C > 0$ such that for every $r \geq 4$ and every bipartite graph $F$ with large enough minimum degree (compared to $r$), we \nolinebreak have
$$\Omega(n^{\frac{c}{\log r}})\leq f_{F,K_r}(n)\leq O(n^{\frac{C}{\log r}}).$$
Note that this is in striking contrast with Ramsey numbers, for which we have $\Omega(n^{\frac{c}{r}})\leq f_{K_2,K_r}(n)\leq O(n^{\frac{C}{r}})$. We also point out that both Theorem \ref{thm:Ktt vs Kr lower bound} and Theorem \ref{thm:Ktt vs Kr upper bound} use methods that are rather novel in the study of Erd\H os--Rogers functions. 

As mentioned above, Mubayi and Verstra\"ete \cite{MV_general_graphs} proved that for all (non-empty) triangle-free graphs $F$, we have $f_{F,K_3}(n)=n^{1/2+o(1)}$. This shows that $f_{F,K_3}(n)$ is quite close to $f_{K_2,K_3}(n)$ for every triangle-free graph $F$. They asked to find an example where the two functions have different orders of magnitudes.

\begin{problem}[Mubayi--Verstra\"ete \cite{MV_general_graphs}] \label{problem:F-free triangle-free}
    Find a triangle-free $F$ for which $f_{F,K_3}(n)/f_{K_2,K_3}(n)\rightarrow \infty$.
\end{problem}

\noindent Note that by the celebrated result of Kim \cite{Kim95} on the Ramsey number $R(3,k)$, we have $f_{K_2,K_3}(n)=\Theta(\sqrt{n \log n})$. Problem \ref{problem:F-free triangle-free} remains open, but in Subsection \ref{sec:F vs triangle} we present a connection to the famous Zarankiewicz problem for $6$-cycles, similar to the connection between Ramsey numbers and the Zarankiewicz problem discussed in \cite{CMMV24}.

\paragraph{Organization of the paper.} In Section \ref{sec:Ktt-free}, we prove Theorems \ref{thm:Ktt vs K4}, \ref{thm:F vs Kr} and \ref{thm:Ktt vs Kr lower bound} and Corollary \ref{cor:large Turan}. In Section \ref{sec:construction}, we prove Theorems \ref{thm:cliquefree vs clique} and \ref{thm:Ktt vs Kr upper bound}. In this section we also discuss the problem of estimating $f_{F,K_3}(n)$ for an arbitrary triangle-free graph $F$, and reveal a connection to the Zarankiewicz problem for $C_6$. In Section \ref{sec:concluding}, we give some concluding remarks.

In Section \ref{sec:Ktt-free}, logarithms are in base $e$, while in Section \ref{sec:construction}, logarithms are in base $2$.

\section{Lower bounds}\label{sec:Ktt-free}
In this section we prove Theorems \ref{thm:Ktt vs K4}, \ref{thm:F vs Kr} and \ref{thm:Ktt vs Kr lower bound} and Corollary \ref{cor:large Turan}.
We denote by $\alpha_F(G)$ the largest order of an $F$-free induced subgraph of $G$.
The {\em $s$-domination number} $\gamma_s(F)$ of a graph $F$ is the minimum $k$ for which there is a set $A \subseteq V(F)$ with $|A| = k$ such that every $v \in V(F) \setminus A$ has at least $s$ neighbours in $A$. We will need the following lemma, showing that graphs of large minimum degree have small $s$-domination number.
\begin{lemma}\label{lem:domination}
Let $t\geq 2$ and let $F$ be a graph with minimum degree $t$. Let $\frac{6\log t}{t} \leq \delta \leq 1$ and set $s = \lfloor \frac{\delta t}{3} \rfloor$. Then $\gamma_s(F) \leq \delta \cdot v(F)$. 
\end{lemma}
\begin{proof}
The assumptions on $\delta$ imply that $t \geq 10$, say. 
Sample a subset $A_0 \subseteq V(F)$ by including each vertex with probability $0.9\delta$ independently. 
For a given $v \in V(F)$, the number of neighbours of $v$ in $A_0$ is distributed $\Bin(d(v),0.9\delta)$. Also, $d(v) \geq t$. 
Thus, by the Chernoff bound, the probability that $|N(v) \cap A_0| < s = \lfloor \frac{\delta t}{3} \rfloor$ is at most 
$\mathbb{P}\left[ \Bin(t,0.9\delta) < \frac{\delta t}{3} \right] \leq e^{-(17/27)^2 \cdot \frac{0.9\delta t}{2}} < e^{-\delta t/6}
\leq e^{-\log t} = \frac{1}{t} \leq \frac{\delta}{10},$ where the last two inequalities use that $\delta \geq \frac{6\log t}{t}$ and the last inequality also uses that $t\geq 10$.

Finally, let $A$ contain all vertices in $A_0$ as well as each $v \in V(F) \setminus A_0$ with less than $s$ neighbours in $A_0$. By linearity of expectation, the expected size of $A$ is at most $0.9\delta \cdot v(F) + 0.1 \delta \cdot v(F) = \delta \cdot v(F)$, as required.     
\end{proof}

The following lemma, which we think is of independent interest, is a key for the proof of Theorems \ref{thm:Ktt vs K4}, \ref{thm:F vs Kr} and \ref{thm:Ktt vs Kr lower bound}. Here and below, for $X \subseteq V(G)$, we let $N(X)$ denote the common neighbourhood of $X$. 
\begin{lemma}\label{claim:sampling}
		Let $0<\delta<\beta<1$, let $F$ be a graph, let $s \geq 1$, and
        suppose that $\gamma_s(F) \leq \delta \cdot v(F)$.
		Let $n$ be sufficiently large and let $G$ be an $n$-vertex graph with $\alpha_F(G) < 0.5n^{\beta-2\delta}$. Then there are at least $0.5 n^{(1-\beta+\delta)s}$ sets $X\subseteq V(G)$ of size $s$ with $|N(X)|\geq n^{1-\beta}$. 
\end{lemma}
\begin{proof}
    We say that an $s$-set $X$ is {\em good} if $|N(X)| \geq n^{1-\beta}$, and {\em bad} otherwise. Suppose by contradiction that there are less than $0.5 n^{(1-\beta+\delta)s}$ good $s$-sets. Sample a set $U$ by including each vertex with probability 
	$p = n^{-1+\beta-\delta}$. Then with high probability, $0.9n^{\beta-\delta} \leq |U| \leq 1.1n^{\beta-\delta}$. Also, the expected number of good $s$-sets contained in $U$ is at most
		$
		0.5n^{(1-\beta+\delta)s} p^s = 0.5,
		$ 
		so with probability at least $0.5$, $U$ contains no good $s$-sets.
		Next, we claim that with high probability, for every bad $s$-set $X$ it holds that $|N(X) \cap U| < q := \lceil \frac{2s}{\delta} \rceil$. Indeed, for a given bad set $X$, the probability that $|N(X) \cap U| \geq q$ is at most 
        $|N(X)|^q \cdot p^q \leq n^{(1-\beta)q}p^q = n^{-\delta q} \leq n^{-2s}$. Since there are at most $\binom{n}{s} \leq n^s$ choices for $X$, our claim follows by the union bound. So we see that with positive probability, $U$ satisfies all of the following properties: $0.9n^{\beta-\delta} \leq |U| \leq 1.1n^{\beta-\delta}$; $U$ contains no good $s$-sets; for every bad $s$-set $X$ it holds that $|N(X) \cap U| < q$. The two latter properties imply that $|N(X) \cap U| < q$ for every $X \subseteq U$ of size $s$. From now on, fix a set $U$ with these properties. 

        Put $f := v(F)$. Let $A \subseteq V(F)$ of size $|A| = \gamma_s(F) \leq \delta f$ such that every $v \in V(F) \setminus A$ has at least $s$ neighbours in $A$. 
        We now bound the number of copies of $F$ in $G[U]$. 
        Clearly, the number of ways to embed $A$ is at most $|U|^{|A|} \leq |U|^{\delta f}$. Now consider a given embedding $\phi : A \rightarrow U$. For each $v \in V(F) \setminus A$, we have $|N_F(v) \cap A| \geq s$. By the properties of $U$, the set $\phi(N_F(v) \cap A)$ has less than $q = O(1)$ common neighbours in $U$. Hence, for each $v \in V(F) \setminus A$, there are $O(1)$ options to embed $v$. It follows that the number of copies of $F$ in $G[U]$ is at most $|U|^{\delta f} \cdot O(1)^f = O(|U|^{\delta f}) < n^{\delta f}$. Now sample a subset $W \subseteq U$ by including each vertex in $W$ with probability $p' = n^{-\delta}$. 
        The expected number of copies of $F$ in $W$ is at most $n^{\delta f} \cdot (p')^f = 1$. Also, the expected size of $W$ is at least $|U| \cdot p' \geq 0.9n^{\beta-\delta} \cdot p' = 0.9n^{\beta - 2\delta}$. Hence, there is an outcome of $W$ for which, by deleting at most one vertex, we obtain an $F$-free induced subgraph of $G$ with at least $0.5n^{\beta-2\delta}$ vertices, a contradiction. 
\end{proof}

\subsection{Proof of Theorems \ref{thm:Ktt vs K4} and \ref{thm:F vs Kr} and Corollary \ref{cor:large Turan}}

We will derive Theorems \ref{thm:Ktt vs K4} and \ref{thm:F vs Kr} from the following theorem.

\begin{theorem}\label{thm:F vs Kr domination}
    Let $r \geq 4$ and let $F$ be a graph which contains a copy of $K_{r-2}$. Let $\delta > 0$, and suppose that $\gamma_s(F) \leq \delta \cdot v(F)$ for $s = \lceil \frac{1}{\delta} \rceil$. Then $f_{F,K_r}(n) \geq 0.5n^{1/2 - 2\delta}$ holds for all sufficiently large~$n$.
\end{theorem}

\begin{proof}
Let $G$ be an $n$-vertex $K_r$-free graph, and suppose by contradiction that $\alpha_F(G) < 0.5n^{1/2 - 2\delta}$. Clearly, we have $\delta<1/4$ by the bound on $\alpha_F(G)$. By Lemma \ref{claim:sampling} with $\beta = 1/2$, there are at least $n^{s/2}$ sets $X$ with $|N(X)| \geq n^{1/2}$. Let us count $(s+2)$-tuples $y,z,x_1,\dots,x_s$ such that $X = \{x_1,\dots,x_s\}$ satisfies $|N(X)| \geq n^{1/2}$, $x_1,\dots,x_s$ are common neighbours of $y,z$, and $yz$ is an edge. 
As $G$ is $K_r$-free, the common neighbourhood of any edge $yz$ is $K_{r-2}$-free and hence $F$-free (as $F$ contains $K_{r-2}$).
Since $\alpha_F(G) < n^{1/2-2\delta}$, this common neighbourhood has size less than $n^{1/2-2\delta}$. Hence, the number of $(s+2)$-tuples as above is less than $2e(G) \cdot n^{s/2-2s\delta}$. 
So by averaging, there exist vertices $x_1,\dots,x_s$ such that for $X = \{x_1,\dots,x_s\}$ we have $|N(X)|\geq n^{1/2}$ and
$e(N(X)) \leq \frac{2e(G) \cdot n^{s/2 - 2s\delta}}{n^{s/2}} \leq n^{2-2s\delta} \leq n^{2 - 2} = 1$. Therefore, $N(X)$ contains an independent set of size $|N(X)| - 1 \geq n^{1/2}-1 \geq n^{1/2-2\delta}$, a contradiction.  
\end{proof}
By combining Theorem \ref{thm:F vs Kr domination} with Lemma \ref{lem:domination}, we get the following. 
\begin{theorem}\label{thm:F vs Kr main}
    Let $r \geq 4$ and let $F$ be a graph with minimum degree $t$ which contains a copy of $K_{r-2}$. Then $f_{F,K_r}(n) \geq 0.5n^{1/2 - 6/\sqrt{t}}$ holds for all sufficiently large $n$.
\end{theorem}
\begin{proof}
    Set $\delta := 3/\sqrt{t}$ and $s = \lceil 1/\delta \rceil$. We may assume that $t \geq 144$, else the conclusion is trivial. In this range we have $\frac{6\log t}{t} \leq \delta \leq 1$. Also, $s \leq \frac{\delta t}{3}$ (by the choice of $\delta$). Hence, by Lemma \ref{lem:domination}, we have $\gamma_s(F) \leq \delta \cdot v(F)$. Now the conclusion follows from Theorem \ref{thm:F vs Kr domination}. 
\end{proof}
Taking $r=4$ in Theorem \ref{thm:F vs Kr main} immediately gives Theorem \ref{thm:Ktt vs K4}. Also, it is easy to see that there exists a graph $F$ which has arbitrarily large minimum degree and contains $K_{r-2}$ but not $K_{r-1}$. For example, we can take the complete $(r-2)$-partite graph with parts of size $t$, where $t$ is sufficiently large. Hence, Theorem~\ref{thm:F vs Kr main} implies Theorem \ref{thm:F vs Kr}.

To deduce Corollary \ref{cor:large Turan}, we will use the following result of Alon, Krivelevich and Sudakov.

\begin{theorem}[Alon--Krivelevich--Sudakov \cite{AKS03}] \label{thm:degeneracy Turan}
    If $F$ is a bipartite graph which does not contain a subgraph of minimum degree at least $t+1$, then $\ex(n,F)=O(n^{2-\frac{1}{4t}})$.
\end{theorem}

\begin{proof}[Proof of Corollary \ref{cor:large Turan}]
    We will prove that $\delta=\eps^2/180$ is a suitable choice. Let $F$ be a bipartite graph with $\ex(m,F)=\Omega(m^{2-\delta})$. Let $t=\lfloor \frac{1}{5\delta}\rfloor$. By Theorem \ref{thm:degeneracy Turan}, $F$ must contain a subgraph of minimum degree at least $t+1$. But then by Theorem \ref{thm:Ktt vs K4}, we have $f_{F,K_4}(n)=\Omega(n^{1/2-6/\sqrt{t+1}})\geq \Omega(n^{1/2-6\sqrt{5\delta}})=\Omega(n^{1/2-\eps})$, as desired.
\end{proof}

\subsection{Proof of Theorem \ref{thm:Ktt vs Kr lower bound}}
    We will derive Theorem \ref{thm:Ktt vs Kr lower bound} from the following theorem.
	
    \begin{theorem}\label{thm:clique vs Ktt} 
       Let $F$ be a graph, let $\delta > 0$ and suppose that $\gamma_s(F) \leq \delta \cdot v(F)$ for $s = \lceil \frac{1}{\delta^3}\rceil$.
       Then for every $k \geq 1$, 
       $f_{F,K_{2^k}}(n) \geq n^{\frac{1}{k} - 2^k\delta}$ for all sufficiently large~$n$.
	\end{theorem}
    Before proving Theorem \ref{thm:clique vs Ktt}, let us use it to prove Theorem \ref{thm:Ktt vs Kr lower bound}.
    \begin{proof}[Proof of Theorem \ref{thm:Ktt vs Kr lower bound}]
    Let $r \geq 4$ and put $k := \lceil \log_2r  \rceil$, so that $r \leq 2^k$. Put $\delta := 2t^{-1/4}$ and $s = \lceil \frac{1}{\delta^3} \rceil$. If $t$ is large enough then $\delta \geq \frac{6\log t}{t}$ and $s \leq \frac{\delta t}{3}$. Then by Lemma \ref{lem:domination}, we have $\gamma_s(F) \leq \delta \cdot v(F)$. Hence, by Theorem \ref{thm:clique vs Ktt}, we have $f_{F,K_r}(n) \geq f_{F,K_{2^k}}(n) \geq n^{\frac{1}{k} - 2^k\delta} \geq n^{\frac{1}{k}-\varepsilon}$, where the last inequality holds if $t$ is large enough. 
    \end{proof}
    
    In the rest of this subsection, we prove Theorem \ref{thm:clique vs Ktt}.
	In the following lemma, $d(U,W)$ stands for the proportion of pairs $(u,w)\in U\times W$ for which $uw$ is an edge.
	\begin{lemma}\label{lem:main}
		Let $\delta,\varepsilon,\beta > 0$, Let $s \geq 2(\frac{\beta}{\varepsilon} + 1)$ be an integer, and let $F$ be a graph with $\gamma_s(F) \leq \delta \cdot v(F)$.
		Let $n$ be sufficiently large and let $G$ be an $n$-vertex graph with $\alpha_F(G) < 0.5n^{\beta-2\delta}$. Then there are 
		$U,W \subseteq V(G)$ with $|U| \geq \Omega(n^{1-\beta - \frac{2\beta(\beta+\varepsilon)}{\varepsilon s}})$, $|W| \geq n^{1-\beta}$ and $d(U,W) \geq n^{-\varepsilon}$. 
	\end{lemma} 
	\begin{proof}
		We may assume that $\delta<\beta<1$, for otherwise the statement of the lemma is trivial. Recall that we call a set $X \subseteq V(G)$ of size $s$ {\em good} if $|N(X)| \geq n^{1-\beta}$. For a set $Y$, let $g(Y)$ denote the number of good $s$-sets containing $Y$. Also, let $g := g(\emptyset)$ be the total number of good $s$-sets. By Lemma \ref{claim:sampling}, we have $g \geq n^{(1-\beta)s}$. 
		Let $1 \leq k \leq s$ be maximal such that there is a set $Y$ of size $k$ satisfying $g(Y) \geq 2^{-k+1} n^{-k} \cdot g$ and $|N(Y)| < n^{1-(k-1)\varepsilon}$. 
		Note that $k$ is well-defined, because by taking $v$ to be a vertex with $g(\{v\}) \geq g/n$ (such a vertex exists by averaging) and setting $Y = \{v\}$, we get $|Y| = 1$, $g(Y) \geq g/n$ and $|N(Y)| < n$. 
		Also, since $g(Y) > 0$, we must have $|N(Y)| \geq n^{1-\beta}$ (because $Y$ is contained in a good $s$-set). Since $|N(Y)| < n^{1-(k-1)\varepsilon}$, it follows that $k < \frac{\beta}{\varepsilon}+1 \leq s/2$. 
		Let $W = N(Y)$, and let $U$ be the set of all vertices $v\in V(G)\setminus Y$ such that 
		$|N(Y \cup \{v\})| \geq |N(Y)|/n^{\varepsilon}$. Suppose first that at least $\frac{1}{2}g(Y)$ of the good $s$-sets $X$ containing $Y$ satisfy that $X \setminus (Y \cup U) \neq \emptyset$. Then, by averaging, there is a vertex $v \notin Y \cup U$ such that there are at least 
		$\frac{1}{2n} \cdot g(Y) \geq 2^{-k} n^{-(k+1)} \cdot g$ good $s$-sets containing $Y \cup \{v\}$. Also, since $v \notin U$, we have
		$|N(Y \cup \{v\})| < |N(Y)|/n^{\varepsilon} < n^{1-k\varepsilon}$. Hence, the set $Y \cup \{v\}$ of size $k+1$ contradicts the maximality of $k$. It follows that at least $\frac{1}{2}g(Y)$ of the good $s$-sets $X$ containing $Y$ satisfy that $X \setminus Y \subseteq U$. 
		Hence, $|U|^{s-k} \geq \binom{|U|}{s-k} \geq \frac{1}{2}g(Y)$.
		On the other hand, 
		$\frac{1}{2}g(Y) \geq 2^{-k} n^{-k} \cdot g \geq 2^{-k} n^{(1-\beta)s - k}$, so we get that 
		$$
		|U| \geq 2^{-\frac{k}{s-k}} \cdot n^{\frac{(1-\beta)s-k}{s-k}} = \Omega(n^{1-\beta - \frac{\beta k}{s-k}}) \geq 
		\Omega(n^{1-\beta - \frac{2\beta(\beta+\varepsilon)}{\varepsilon s}}),
		$$
		using that $k \leq \frac{\beta}{\varepsilon} + 1 \leq s/2$. 
		Also, $|W| \geq n^{1-\beta}$ because $W = N(Y)$ and $g(Y) > 0$. 
		Finally, $|N(v)\cap W| \geq |W|/n^{\varepsilon}$ for every $v \in U$ by the definition of $U$ and as $W = N(Y)$. Hence,
		$d(U,W) \geq n^{-\varepsilon}$. This completes the proof of the lemma. 
	\end{proof}
	\begin{proof}[Proof of Theorem \ref{thm:clique vs Ktt}] 
        We prove the theorem by induction on $k$. The case $k=1$ is trivial, so let $k \geq 2$. 
        We may assume that $\delta < 2^{-k}/k$, else the assertion is trivial. 
        Let $G$ be an $n$-vertex $K_{2^k}$-free graph. Our goal is to show that $G$ has an $F$-free induced subgraph on at least $n^{\frac{1}{k} - 2^k\delta}$ vertices. 
		If this does not hold, then by Lemma \ref{lem:main} with $\beta := \frac{1}{k}$ and $\varepsilon := \delta^2$, there are subsets $U,W \subseteq V(G)$ with $|U| \geq \Omega(n^{1-\beta - \frac{2\beta(\beta+\varepsilon)}{\varepsilon s}}) \geq n^{1-\frac{1}{k} - \frac{1}{\varepsilon s}} \geq n^{1 - \frac{1}{k} - \delta}$, $|W| \geq n^{1-\frac{1}{k}}$, and $d(U,W) \geq n^{-\varepsilon} = n^{-\delta^2}$. 
        We may apply Lemma \ref{lem:main} because $s = \lceil \frac{1}{\delta^3} \rceil \geq \frac{1}{\varepsilon} + 2 \geq 2(\frac{\beta}{\varepsilon} + 1)$.
        Next, we prove the following claim using the so-called dependent random choice method (see, e.g., \cite{FoxSudakov} for a detailed description of this method). 
        
		\begin{claim}\label{claim:DRC}
			There is a subset $A \subseteq W$ of size $|A| > n^{1-\frac{1}{k} - 2^{k-1}\delta} - 1$ such that every $2^{k-1}$ vertices in $A$ have at least $|U| \cdot n^{-2\delta}$ common neighbours.
		\end{claim}
		\begin{proof}
			Sample 
			$q := \lfloor \frac{2^{k-1}}{\delta} \rfloor$ vertices 
			$u_1,\dots,u_q \in U$ uniformly at random and independently (with repetitions allowed), and let $N = N(\{u_1,\dots,u_q\}) \cap W$. Since $d(U,W) \geq n^{-\delta^2}$, by convexity we get
			$$
			\mathbb{E}[|N|] \geq |W| \cdot n^{-q\delta^2} \geq |W| \cdot n^{-2^{k-1}\delta} \geq
			n^{1-\frac{1}{k} - 2^{k-1}\delta}.
			$$
			Also, for each subset $K \subseteq W$ of size $|K| = 2^{k-1}$ with $|N(K) \cap U| \leq |U| \cdot n^{-2\delta}$, we have 
			$$
			\mathbb{P}[K \subseteq N] \leq n^{-2q\delta} \leq n^{-2^{k-1}} \leq |W|^{-2^{k-1}}.
			$$
			Hence, the expected number of subsets $K \subseteq N$ of size $2^{k-1}$ with $|N(K) \cap U| \leq |U| \cdot n^{-2\delta}$ is less than~$1$. 
            Delete one vertex from each such subset to obtain $A$.
            The claim follows by linearity \nolinebreak of \nolinebreak expectation. 
		\end{proof}
		Let $A$ be the subset given by Claim \ref{claim:DRC}. If $G[A]$ is $K_{2^{k-1}}$-free, then by the induction hypothesis, $G[A]$ contains an $F$-free induced subgraph of order
		$$
		|A|^{\frac{1}{k-1} - 2^{k-1}\delta} \geq \left( 0.5n^{1 - \frac{1}{k} - 2^{k-1}\delta} \right)^{\frac{1}{k-1} - 2^{k-1}\delta} \geq 
		n^{\frac{1}{k} - 2^{k-1} \delta - 2^{k-1}\delta} = n^{\frac{1}{k} - 2^k\delta},
		$$ 
		as required. 
		
		Suppose now that $G[A]$ contains a clique $K$ of size $2^{k-1}$. Let $B$ be the common neighbourhood of $K$, so $|B| \geq |U| n^{-2\delta} \geq 
		n^{1 - \frac{1}{k} - 3\delta}$. Since $G$ is $K_{2^k}$-free, $G[B]$ is $K_{2^{k-1}}$-free. Hence, by the induction hypothesis, $G[B]$ contains an $F$-free induced subgraph of order
		$$
		|B|^{\frac{1}{k-1} - 2^{k-1}\delta} \geq 
		\left( n^{1 - \frac{1}{k} - 3\delta} \right)^{\frac{1}{k-1} - 2^{k-1}\delta} \geq 
        n^{\frac{1}{k} - 2^k\delta},
		$$ 
		as required. This completes the proof. 
	\end{proof}

\section{Upper bound constructions} \label{sec:construction}

\subsection{Proof of Theorem \ref{thm:Ktt vs Kr upper bound}}
It suffices to prove Theorem \ref{thm:Ktt vs Kr upper bound} for $F = K_{t,t}$ (since every bipartite graph is contained in $K_{t,t}$ for a sufficiently large $t$). Hence, Theorem \ref{thm:Ktt vs Kr upper bound} follows from the following result.

 \begin{theorem}\label{thm:Ktt vs Kr construction}
     There is an absolute constant $C$ such that for every $k\geq 2$ and every $t$ we have $f_{K_{t,t},K_{2^k}}(n)=O(n^{C/k})$.
 \end{theorem}

	For graphs $G,H$, the {\em lexicographic product} $G \cdot H$ is the graph obtained from $G$ by substituting a copy of $H$ for each vertex of $G$ (and replacing edges of $G$ with complete bipartite graphs). It is easy to see that $\omega(G \cdot H) = \omega(G) \cdot \omega(H)$ and $\chi(G \cdot H) \leq \chi(G) \cdot \chi(H)$.
	\begin{lemma}\label{lem:substitution 2 graphs}
		For any positive integer $t$ and graphs $G$ and $H$, we have 
        $$\alpha_{K_{t,t}}(G \cdot H) \leq \alpha(G)\alpha_{K_{t,t}}(H) + (t-1)\alpha_{K_{t,t}}(G).$$ 
        In particular, $\alpha_{K_{t,t}}(G\cdot G)\leq t\alpha(G)\alpha_{K_{t,t}}(G)$.
	\end{lemma}
	\begin{proof}
		For each $v \in V(G)$, let $A_v$ be the vertex set of the copy of $H$ corresponding to $v$. Let $X \subseteq V(G\cdot H)$ such that $(G\cdot H)[X]$ is $K_{t,t}$-free. 
		Let $I = \{v\in V(G) : |A_v \cap X| \geq t\}$ and let $J = \{v\in V(G) : 0 < |A_v \cap X| \leq t-1\}$. Then $I$ is an independent set in $G$, and $G[J]$ is $K_{t,t}$-free. Also, for each $v \in I$, $H[A_v \cap X]$ is $K_{t,t}$-free. It follows that 
		$|X| \leq |I| \cdot \alpha_{K_{t,t}}(H) + (t-1)|J|$, implying the lemma. 
	\end{proof}
	\noindent

    We construct $K_{2^k}$-free graphs with no large $K_{t,t}$-free set by induction on $k$. Roughly speaking, we start with a $K_{2^{k/2}}$-free graph $G_0$ with no large $K_{t,t}$-free set,
     take a union of it with a random graph on the same vertex set to obtain a graph $H$ (where the random graph ensures that $H$ has small independence number), and then consider $H\cdot H$. Then $H$ has no large $K_{t,t}$-free set by Lemma~\ref{lem:substitution 2 graphs}. Unfortunately, $H$ may contain a clique of size significantly greater than $2^{k/2}$, which means that $H\cdot H$ may contain a clique of size significantly greater than $2^k$.

	In order to overcome this issue, instead of considering the clique number, we consider the property of having no subgraph on $O(1)$ vertices with large chromatic number. This is more convenient because the chromatic number of the union of two graphs is at most the product of their chromatic numbers, whereas the clique number can be exponential in the clique numbers.

 \begin{definition}
     For an integer $r\geq 3$, let $S_r$ be the set of all $\rho\geq 0$ with the property that for all positive integers $t,s$ there is some $n_0=n_0(\rho,r,t,s)$ such that for all $n\geq n_0$ there exists an $n$-vertex graph $G$ in which every subgraph on $s$ vertices is $(r-1)$-colourable and which has $\alpha_{K_{t,t}}(G)\leq n^{\rho}$.

     Furthermore, let $\rho_r=\textrm{inf}(S_r)$.
 \end{definition}
\noindent 
Note that $1\in S_r$, so $\rho_r$ is well-defined and $\rho_r\leq 1$.

\begin{lemma} \label{lem:product chromatic}
    Let $H$ be a graph (on at least $s$ vertices) in which every subgraph on $s$ vertices is $r$-colourable. Then every subgraph of $H\cdot H$ on $s$ vertices is $r^2$-colourable.
\end{lemma}

\begin{proof}
    Any subgraph of $H\cdot H$ on $s$ vertices is the subgraph of $H[X]\cdot H[Y]$ for some $X,Y\subset V(H)$ of size $s$. But $\chi(H[X]\cdot H[Y])\leq \chi(H[X])\cdot \chi(H[Y])\leq r^2$.
\end{proof}
 
	We also need the following well-known properties of random graphs, which can be easily proved using the union bound.
	\begin{lemma}\label{lem:random graph properties}
		Let $s$ and $r$ be fixed positive integers. Let $p=n^{-2/r}/\log n$. Then $G \sim G(n,p)$ satisfies the following properties.
		\begin{enumerate}
			\item Almost surely $\alpha(G) \leq n^{2/r}(\log n)^3$. 
			\item Almost surely every subgraph of $G$ on at most $s$ vertices has a vertex of degree at most $r-1$. Hence, every such subgraph is $r$-colorable. 
		\end{enumerate}
	\end{lemma}

The following lemma establishes a recursive inequality for the numbers $\rho_r$, which we will then use to prove Theorem \ref{thm:Ktt vs Kr construction}.
\begin{lemma}\label{lem:rho recursion}
	For every $1\leq i\leq k/2$, we have
	$$\rho_{2^k}\leq \frac{1}{2}\rho_{2^i}+2^{i-\lfloor k/2 \rfloor}.$$
\end{lemma}

\begin{proof}
    Let $\eps>0$. It suffices to prove that $\frac{1}{2}\rho_{2^i}+2^{i-\lfloor k/2 \rfloor}+\eps\in S_{2^k}$.

    Let $s$ and $t$ be positive integers.
    By the definition of $\rho_{2^i}$, there exists some $n_0=n_0(\eps,i,t,s)$ such that for all $n\geq n_0$ there is an $n$-vertex graph $G_0$ with the property that $\alpha_{K_{t,t}}(G_0) \leq n^{\rho_{2^i}+\eps}$ and every subgraph of $G_0$ on $s$ vertices has chromatic number at most $2^i-1$. By Lemma \ref{lem:random graph properties} (applied with $r=2^{\lfloor k/2\rfloor-i}$), there exists some $n_1=n_1(i,k,s)$ such that if $n\geq n_1$, then there is a graph $G_1$ on $n$ vertices such that $\alpha(G_1) \leq n^{2^{i+1-\lfloor k/2\rfloor}}(\log n)^3$ and every subgraph of $G_1$ on $s$ vertices has chromatic number at most $2^{\lfloor k/2 \rfloor-i}$. For $n\geq \max(n_0,n_1)$, we can find $G_0$ and $G_1$ on the same vertex set and let $H$ be the union of $G_0$ and $G_1$. Then $\alpha(H) \leq \alpha(G_1)\leq n^{2^{i+1-\lfloor k/2\rfloor}}(\log n)^3$, $\alpha_{K_{t,t}}(H)\leq \alpha_{K_{t,t}}(G_0) \leq n^{\rho_{2^i}+\eps}$, and every subgraph of $H$ on $s$ vertices has chromatic number at most $(2^i-1)\cdot 2^{\lfloor k/2\rfloor-i} \leq 2^{\lfloor k/2\rfloor} - 1$. Set $G := H \cdot H$. Then, by Lemma \ref{lem:product chromatic}, every subgraph of $G$ on $s$ vertices has chromatic number at most $(2^{\lfloor k/2\rfloor} - 1)^2 \leq 2^k-1$. Also, $N := |V(G)| = n^{2}$, and by Lemma \ref{lem:substitution 2 graphs} we have
	$$
	\alpha_{K_{t,t}}(G) \leq t \cdot \alpha(H) \cdot \alpha_{K_{t,t}}(H) \leq t n^{2^{i+1-\lfloor k/2\rfloor} + \rho_{2^i}+\eps}(\log n)^3 \leq t N^{2^{i-\lfloor k/2\rfloor}+\frac{1}{2}\rho_{2^i}+\eps/2}(\log N)^3.
	$$ 
 Thus, there exists some $N_0=N_0(\eps,k,t,s)$ such that for all perfect squares $N\geq N_0$ there is an $N$-vertex graph $G$ with the property that $\alpha_{K_{t,t}}(G)\leq N^{2^{i-\lfloor k/2\rfloor} + \frac{1}{2}\rho_{2^i}+3\eps/4}$ and every subgraph of $G$ on $s$ vertices has chromatic number at most $2^k-1$.
 
 This clearly implies that $2^{i-\lfloor k/2\rfloor} + \frac{1}{2}\rho_{2^i}+\eps\in S_{2^k}$, as desired.
\end{proof}

\begin{proof}[Proof of Theorem \ref{thm:Ktt vs Kr construction}]
It suffices to show that $\rho_{2^k} \leq \frac{C}{k}$, where $C$ is a large enough constant to be chosen later. 
Indeed, by the definition of $\rho_{2^k}$, if $\rho=2\rho_{2^k}$, then $\rho\in S_{2^k}$, so for every large enough~$n$ there exists an $n$-vertex graph $G$ with $\alpha_{K_{t,t}}(G) \leq n^{\rho}$, and such that every subgraph of $G$ on $2^k$ vertices is $(2^k-1)$-colorable, which implies that $G$ is $K_{2^k}$-free. 

To facilitate induction, we will actually prove that (say) 
\begin{equation*}\label{eq:rho induction}
\rho_{2^k} \leq \frac{C}{k}\left( 1-k^{-1/3} \right)
\end{equation*}
for all $k \geq 2$.
If $k \leq C/5$ then trivially $\rho_{2^k} \leq 1 \leq 
\frac{C/5}{k} \leq \frac{C}{k}(1-k^{-1/3})$. Suppose now that $k > C/5$. Taking $i=\lfloor k/2-\sqrt{k} \rfloor$, we have by Lemma \ref{lem:rho recursion} that
\begin{align}\label{eq:rho}
\nonumber \rho_{2^k}\leq \frac{1}{2}\rho_{2^i}+2^{i-\lfloor k/2 \rfloor}
&\leq \frac{1}{2} \cdot \frac{C}{i}\left(1 - i^{-1/3}\right) + 2^{i-\lfloor k/2 \rfloor} \\ \nonumber &\leq 
\frac{C}{k - 2\sqrt{k} - 2} \left(1 - (k/2)^{-1/3}\right) + 2^{1-\sqrt{k}} \\ \nonumber &\leq 
C\left( \frac{1}{k} + \frac{3}{k^{3/2}} \right) \left(1-(k/2)^{-1/3}\right) + 2^{1-\sqrt{k}}
\\ &\leq 
\frac{C}{k}\left(1 - (k/2)^{-1/3}\right) + \frac{3C}{k^{3/2}} + 2^{1-\sqrt{k}}
,
\end{align}
where the second inequality holds by the induction hypothesis, the third inequality follows from $k/2 -\sqrt{k} - 1 \leq i \leq k/2 - \sqrt{k}$, and the fourth inequality holds for large enough $k$. So it suffices to show that the right-hand side of \eqref{eq:rho} is at most $\frac{C}{k}(1 - k^{-1/3})$. This is equivalent to having 
\begin{equation}
\frac{3C}{k^{3/2}} + 2^{1-\sqrt{k}} \leq 
\frac{C}{k} \left( (k/2)^{-1/3} - k^{-1/3} \right). \label{eqn:simplified}
\end{equation}
As $(1/2)^{-1/3} - 1 \geq 1/4$, the right-hand side of \eqref{eqn:simplified} is at least $\frac{C}{4k^{4/3}}$, so the required inequality holds for $k \geq C/5$, provided that $C$ is large enough. 
\end{proof}

\subsection{The proof of Theorem \ref{thm:cliquefree vs clique}}

In this subsection we prove Theorem \ref{thm:cliquefree vs clique} in the following more precise form.

\begin{theorem} \label{thm:cliquefree vs clique precise}
    For every $r\geq 4$ and every $K_{r-1}$-free graph $F$ on $s\geq 2$ vertices, we have $f_{F,K_r}(n)=O(n^{1/2-\frac{1}{8s-10}}(\log n)^3)$.
\end{theorem}

\begin{remark}
    In fact, our construction is such that every set of size roughly $n^{1/2-\frac{1}{8s-10}}(\log n)^3$ contains an \emph{induced} copy of $F$.
\end{remark}

The proof of Theorem \ref{thm:cliquefree vs clique precise} uses the method from \cite{Janzer_Sudakov} (which in turn built on \cite{Mattheus_Verstraete}), where this result was proved in the special case $F=K_s$, $r=s+2$. Similarly to those papers, the following graph provides the starting point in our construction.

\begin{proposition}[\cite{ONan72} or \cite{Mattheus_Verstraete}] \label{prop:algebraic graph}
    For every prime $q$, there is a bipartite graph $K$ with vertex sets $X$ and $Y$ such that the following hold.
    \begin{enumerate}
        \item $|X|=q^4-q^3+q^2$ and $|Y|=q^3+1$.
        \item $d_K(x)=q+1$ for every $x\in X$ and $d_K(y)=q^2$ for every $y\in Y$.
        \item $K$ is $C_4$-free. \label{prop:C4-free}
        \item $K$ does not contain the subdivision of $K_4$ as a subgraph with the part of size $4$ embedded to $X$. \label{prop:subdivisionfree}
    \end{enumerate}
\end{proposition}

Throughout this subsection, let $r\geq 4$ be a fixed positive integer and let $F$ be a fixed $K_{r-1}$-free graph on $s$ vertices. Let us identify the vertex set of $F$ with $[s]$. Let $q$ be a prime and let $K$ be the graph provided by Proposition~\ref{prop:algebraic graph}. We now construct a $K_r$-free graph $H$ on vertex set $X$ randomly as follows. 
For each $y\in Y$, partition $N_K(y)$ uniformly randomly as $A_1(y)\cup A_2(y)\cup \dots \cup A_s(y)$ and place a complete bipartite graph between $A_i(y)$ and $A_j(y)$ whenever $i$ and $j$ are adjacent in $F$. In other words, we place a blow-up of $F$ in $N_K(y)$ with parts $A_1(y),\dots,A_s(y)$. The following lemma, proved in \cite{Janzer_Sudakov}, combined with properties \ref{prop:C4-free} and \ref{prop:subdivisionfree} of Proposition \ref{prop:algebraic graph}, shows that $H$ is $K_r$-free with probability 1.

\begin{lemma}[{\cite[Lemma 2.2]{Janzer_Sudakov}}] \label{lem:clique partition}
    Assume that the edge set of a $K_r$ is partitioned into cliques $C_1,\dots,C_k$ of size at most $r-2$. Then there exist four vertices such that all six edges between them belong to different cliques $C_i$.
\end{lemma}

To see that Lemma \ref{lem:clique partition} implies that $H$ is $K_r$-free, assume that $H$ does contain a copy of $K_r$ on vertex set $R$. Note that by property \ref{prop:C4-free} of Proposition \ref{prop:algebraic graph}, for any edge $uv$ in the complete graph $H[R]$, there is a unique $y\in Y$ such that $u,v\in N_K(y)$. Hence, we can partition the edge set of $H[R]$ into cliques, one with vertex set $N_K(y)\cap R$ for each $y\in Y$ such that $|N_K(y)\cap R|\geq 2$. Moreover, any such clique has size at most $r-2$. (Indeed, $F$ is $K_{r-1}$-free, so if $|N_K(y)\cap R|\geq r-1$, then $N_K(y)\cap R$ must contain distinct vertices $u\in A_i(y)$ and $v\in A_j(y)$ such that $ij\not \in E(F)$ (or $i=j$), meaning that $uv$ is not an edge in $H$.) Hence, by Lemma \ref{lem:clique partition}, there are four vertices in $R$ such that for any two of them there is a different common neighbour in $Y$ in the graph $K$, contradicting property \ref{prop:subdivisionfree} of Proposition~\ref{prop:algebraic graph}.

Our key lemma, proved in Section \ref{sec:Ffree sets}, is as follows. Here and below we ignore floor and ceiling signs whenever they are not crucial.

\begin{lemma} \label{lem:few F-free}
    Let $q$ be a sufficiently large prime and let $t=q^{2-\frac{1}{s-1}}(\log q)^{3}$. Then with positive probability the number of sets $T\subset X$ of size $t$ for which $H[T]$ is $F$-free is at most $(q^{\frac{1}{s-1}})^t$.
\end{lemma}

It is easy to deduce Theorem \ref{thm:cliquefree vs clique precise} from this.

\begin{proof}[Proof of Theorem \ref{thm:cliquefree vs clique precise}]
    Take an outcome of $H$ which satisfies the conclusion of Lemma \ref{lem:few F-free}. Let $\tilde{X}$ be a random subset of $X$ obtained by keeping each vertex independently with probability $q^{-1/(s-1)}$, and let $G_0=H[\tilde{X}]$. Then for each set $T\subset X$, the probability that $T\subset \tilde{X}$ is $(q^{-1/(s-1)})^{|T|}$. Hence, for $t=q^{2-1/(s-1)}(\log q)^{3}$, Lemma \ref{lem:few F-free} implies that the expected number of $F$-free sets of size $t$ in $G_0$ is at most 1. Removing one vertex from each such set, we obtain a $K_r$-free graph $G$ in which every vertex set of size $t$ contains a copy of $F$. The expected number of vertices in $G$ is at least $|X|q^{-1/(s-1)}-1\geq \frac{1}{2}q^{4-1/(s-1)}-1$, so there exists an outcome for $G$ with at least $\frac{1}{2}q^{4-1/(s-1)}-1$ vertices. So, for each sufficiently large prime $q$, there is a $K_r$-free graph with at least $\frac{1}{2}q^{(4s-5)/(s-1)}-1$ vertices in which every vertex set of size $q^{(2s-3)/(s-1)}(\log q)^{3}$ contains a copy of $F$. Using Bertrand's postulate, this implies that $f_{F,K_r}(n)=O(n^{\frac{2s-3}{4s-5}}(\log n)^{3})=O(n^{1/2-\frac{1}{8s-10}}(\log n)^3)$, completing the proof.
\end{proof}

\subsubsection{The number of $F$-free sets} \label{sec:Ffree sets}

In this subsection we prove Lemma \ref{lem:few F-free}. While the proof is very similar to that of Lemma 2.3 in \cite{Janzer_Sudakov}, there are some small necessary changes, and we include a full proof for completeness. We will use the following lemma from \cite{Janzer_Sudakov}.

\begin{lemma}[{\cite[Lemma 2.4]{Janzer_Sudakov}}] \label{lem:thereisgoodscale}
    Assume that $q$ is sufficiently large. Then with positive probability, for every $U\subset X$ with $|U|\geq 500s^2q^2$ there exists some $\gamma\geq |U|/q^2$ such that the number of $y\in Y$ with $\gamma/(10s)\leq|A_i(y)\cap U|\leq \gamma$ for all $i\in [s]$ is at least $|U|q/(8(\log q)\gamma)$.
\end{lemma}

\begin{definition}
    Let us call an instance of $H$ \emph{nice} if it satisfies the conclusion of Lemma \ref{lem:thereisgoodscale}.
\end{definition}

Lemma \ref{lem:few F-free} can now be deduced from the following.

\begin{lemma} \label{lem:few F-free if nice}
    Let $q$ be sufficiently large and let $t=q^{2-1/(s-1)}(\log q)^{3}$. If $H$ is nice, then the number of sets $T\subset X$ of size $t$ for which $H[T]$ is $F$-free is at most $(q^{1/(s-1)})^t$.
\end{lemma}

In what follows, we will consider an $s$-uniform hypergraph on vertex set $X$ whose hyperedges correspond to the copies of $F$ in $H$. Then $F$-free subsets of $X$ will correspond to independent sets in this hypergraph, so to prove Lemma \ref{lem:few F-free if nice}, it suffices to bound the number of independent sets of certain size. This will be achieved using the hypergraph container method. For an $s$-uniform hypergraph $\mathcal{G}$ and some $\ell \in [s]$, we write $\Delta_{\ell}(\mathcal{G})$ for the maximum number of hyperedges in $\mathcal{G}$ containing the same set of $\ell$ vertices.

We use the following result from \cite{Janzer_Sudakov}.

\begin{lemma}[{\cite[Corollary 2.8]{Janzer_Sudakov}}] \label{lem:BMScontainer}
    For every positive integer $s\geq 2$ and positive reals $p$ and $\lambda$, the following holds. Suppose that $\mathcal{G}$ is an $s$-uniform hypergraph with at least two vertices such that $pv(\mathcal{G})$ and $v(\mathcal{G})/\lambda$ are integers, and for every $\ell\in [s]$,
    $$\Delta_{\ell}(\mathcal{G})\leq \lambda\cdot p^{\ell-1}\frac{e(\mathcal{G})}{v(\mathcal{G})}.$$

    Then there exists a collection $\mathcal{C}$ of at most $v(\mathcal{G})^{spv(\mathcal{G})}$ sets of size at most $(1-\delta \lambda^{-1})v(\mathcal{G})$ such that for every independent set $I$ in $\mathcal{G}$, there exists some $R\in \mathcal{C}$ with $I\subset R$, where $\delta=2^{-s(s+1)}$.
\end{lemma}

Let $\mathcal{H}$ be the $s$-uniform hypergraph on vertex set $X$ in which $s$ vertices form a hyperedge if they induce a copy of $F$ in $H$. The next lemma shows that if $H$ is nice, then a suitable subgraph of~$\mathcal{H}$ (chosen with the help of Lemma \ref{lem:thereisgoodscale}) satisfies the codegree conditions in Lemma \ref{lem:BMScontainer} with small values of $\lambda$ and $p$.

\begin{lemma} \label{lem:bounded degree}
    Assume that $H$ is nice. Then for each $U\subset X$ of size at least $500s^2q^2$ there exists a subgraph $\mathcal{G}$ of $\mathcal{H}[U]$ (on vertex set $U$) which satisfies
    \begin{equation}
        \Delta_{\ell}(\mathcal{G})\leq \lambda\cdot p^{\ell-1}\frac{e(\mathcal{G})}{v(\mathcal{G})} \label{eqn:bounded codegrees}
    \end{equation}
    for every $\ell\in [s]$ with $\lambda=O_s(\log q)$ and $p\leq |U|^{-1}q^{2-1/(s-1)}$.
\end{lemma}

\begin{proof}
    Since $H$ is assumed to be nice, there exists some $\gamma\geq |U|/q^2$ such that the number of $y\in Y$ with $\gamma/(10s)\leq|A_i(y)\cap U|\leq \gamma$ for all $i\in [s]$ is at least $|U|q/(8(\log q)\gamma)$. Let $p=(\gamma q^{\frac{1}{s-1}})^{-1}\leq |U|^{-1}q^{2-1/(s-1)}$. Let $E(\mathcal{G})$ consist of all $s$-sets $\{x_1,x_2,\dots,x_s\}$ in $U$ for which there exists $y\in Y$ with $\gamma/(10s)\leq|A_i(y)\cap U|\leq \gamma$ and $x_i\in A_i(y)\cap U$ for all $i\in [s]$. Clearly, such vertices $x_1,x_2,\dots,x_s$ induce a copy of $F$ in $H$, so $\mathcal{G}$ is indeed a subgraph of $\mathcal{H}$.

    It remains to verify the codegree condition (\ref{eqn:bounded codegrees}).  Roughly speaking, the codegrees are small because for any set $S$ of at least two vertices in $U$, there is at most one vertex $y\in Y$ in the common neighbourhood of $S$ (since $K$ is $C_4$-free), and then all hyperedges in $\mathcal{G}$ containing $S$ live entirely in $N_K(y)$. More precisely, as we are only using those vertices $y\in Y$ to define hyperedges in $\mathcal{G}$ which satisfy $|A_i(y)\cap U|\leq \gamma$ for all $i$, we have $\Delta_{\ell}(\mathcal{G})\leq \gamma^{s-\ell}$ for each $2\leq \ell\leq s$. 
    Moreover, as $d_K(x)=q+1$ for all $x\in X$, we have $\Delta_1(\mathcal{G})\leq (q+1)\gamma^{s-1}$.

    On the other hand, note that $e(\mathcal{G})\geq \frac{|U|q}{8(\log q)\gamma}\cdot (\frac{\gamma}{10s})^s=\Omega_s(|U|q\gamma^{s-1}/\log q)$, so $e(\mathcal{G})/v(\mathcal{G})=\Omega_s(q\gamma^{s-1}/\log q)$. It follows that if $\lambda=C\log q$ for a sufficiently large constant $C=C(s)$, then $\lambda\cdot p^{\ell-1}\frac{e(\mathcal{G})}{v(\mathcal{G})}\geq 2q^{1-(\ell-1)/(s-1)}\gamma^{s-\ell}$. Hence, (\ref{eqn:bounded codegrees}) holds for each $1\leq \ell\leq s$.
\end{proof}

Combining Lemma \ref{lem:BMScontainer} and Lemma \ref{lem:bounded degree}, we prove the following result.

\begin{lemma} \label{lem:container}
    Let $q$ be sufficiently large and assume that $H$ is nice. Let $U$ be a subset of $X$ of size at least $500s^2q^2$. Now there exists a collection $\mathcal{C}$ of at most $(q^4)^{sq^{2-1/(s-1)}}$ sets of size at most $(1-\Omega_s((\log q)^{-1}))|U|$ such that for any $F$-free (in $H$) set $T\subset U$ there exists some $R\in \mathcal{C}$ with $T\subset R$.
\end{lemma}

\begin{proof}
    Choose a hypergraph $\mathcal{G}$ and parameters $\lambda,p$ according to Lemma \ref{lem:bounded degree}. By Lemma \ref{lem:BMScontainer}, there exists a collection $\mathcal{C}$ of at most $|U|^{sp|U|}$ sets of size at most $(1-2^{-s(s+1)}\lambda^{-1})|U|$ such that for every independent set $I$ in $\mathcal{G}$, there exists some $R\in \mathcal{C}$ such that $I\subset R$. The lemma follows by noting that any $F$-free set is an independent set in $\mathcal{G}$, $|U|\leq q^4$, $p\leq |U|^{-1}q^{2-1/(s-1)}$ and $\lambda=O_s(\log q)$.
\end{proof}

\begin{corollary} \label{cor:few sets}
    Let $q$ be sufficiently large and assume that $H$ is nice. Then there is a collection $\mathcal{C}$ of at most $(q^4)^{O_s(q^{2-1/(s-1)}(\log q)^2)}$ sets of size at most $500s^2q^2$ such that for any $F$-free (in $H$) set $T\subset X$ there exists some $R\in \mathcal{C}$ such that $T\subset R$.
\end{corollary}

\begin{proof}
    By Lemma \ref{lem:container}, there exists a positive constant $c_s$ such that whenever $U$ is a subset of $X$ of size at least $500s^2q^2$, then there is a collection $\mathcal{D}$ of at most $(q^4)^{sq^{2-1/(s-1)}}$ sets of size at most $(1-c_s(\log q)^{-1})|U|$ such that for any $F$-free set $T\subset U$ there exists some $R\in \mathcal{D}$ with $T\subset R$.

    We prove by induction that for each positive integer $j$ there is a collection $\mathcal{C}_j$ of at most $(q^4)^{jsq^{2-1/(s-1)}}$ sets of size at most $\max\left(500s^2q^2,(1-c_s(\log q)^{-1})^j|X|\right)$ such that for any $F$-free set $T\subset X$ there exists some $R\in \mathcal{C}_j$ with $T\subset R$. Note that, since $|X| \leq q^4$, by choosing $j$ to be a suitable integer of order $\Theta_s((\log q)^{2})$, the corollary follows. The base case $j=1$ is immediate by the first paragraph (applied in the special case $U=X$).

    Let now $\mathcal{C}_j$ be a suitable collection for $j$ and define $\mathcal{C}_{j+1}$ as follows. For each $U\in \mathcal{C}_j$ of size greater than $500s^2q^2$, take a collection $\mathcal{D}(U)$ of at most $(q^4)^{sq^{2-1/(s-1)}}$ sets of size at most $(1-c_s(\log q)^{-1})|U|$ such that for any $F$-free set $T\subset U$ there exists some $R\in \mathcal{D}(U)$ with $T\subset R$. Let $$\mathcal{C}_{j+1}=\{U\in \mathcal{C}_j:|U|\leq 500s^2q^2\}\cup \bigcup_{U\in \mathcal{C}_j:|U|>500s^2q^2} \mathcal{D}(U).$$
    Clearly, $|\mathcal{C}_{j+1}|\leq |\mathcal{C}_j|(q^4)^{sq^{2-1/(s-1)}}\leq (q^4)^{(j+1)sq^{2-1/(s-1)}}$.
    \sloppy Moreover, since every set in $\mathcal{C}_j$ has size at most $\max\left(500s^2q^2,(1-c_s(\log q)^{-1})^j|X|\right)$, it follows that any set in $\mathcal{C}_{j+1}$ has size at most $\max\left(500s^2q^2,(1-c_s(\log q)^{-1})^{j+1}|X|\right)$. Finally, for any $F$-free set $T\subset X$ there exists some $U\in \mathcal{C}_j$ with $T\subset U$ and hence there exists some $R\in \mathcal{C}_{j+1}$ (either $U$ or some element of $\mathcal{D}(U)$) such that $T\subset R$. This completes the induction step and the proof.
\end{proof}

Corollary \ref{cor:few sets} implies that if $q$ is sufficiently large and $H$ is nice, then the number of $F$-free sets of size $t=q^{2-1/(s-1)}(\log q)^{3}$ in $H$ is at most $$(q^4)^{O_s(q^{2-1/(s-1)}(\log q)^{2})}\binom{500s^2q^2}{t}\leq (q^4)^{O_s(q^{2-1/(s-1)}(\log q)^{2})}(q^{1/(s-1)}/\log q)^t\leq (q^{1/(s-1)})^t,$$
proving Lemma \ref{lem:few F-free if nice}.

\subsection{$F$-free induced subgraphs in triangle-free graphs} \label{sec:F vs triangle}

In this subsection, we observe a connection between Problem \ref{problem:F-free triangle-free} and the Zarankiewicz problem for $6$-cycles.

Let $z(n,m,\{C_4,C_6\})$ denote the maximum number of edges in a bipartite graph with $n+m$ vertices which does not contain $C_4$ or $C_6$ as a subgraph. An old result of de Caen and Sz\'ekely~\cite{DSz97} states that $z(n,m,\{C_4,C_6\})=O(n^{2/3}m^{2/3})$ for $n^{1/2}\leq m\leq n^2$. They observed that there are matching constructions for $m=n$, $m=n^{7/8}$, $m=n^{4/5}$ and $m=n^{1/2}$, but that there is some function $h(n)\rightarrow \infty$ such that $z(n,m,\{C_4,C_6\})=o(n^{2/3}m^{2/3})$ holds for $\omega(n^{1/2})\leq m\leq n^{1/2}h(n)$. We note that $h(n)$ comes from an application of the Ruzsa--Szemer\'edi $(6,3)$-theorem \cite{RSz78} and is of order $e^{\Theta(\log^*(n))}$, where $\log^*(n)$ is the iterated logarithm function.

Roughly speaking, we prove that if $z(n,m,\{C_4,C_6\})=\Theta(n^{2/3}m^{2/3})$ for $m\approx n^{1/2}(\log n)^{3/2}$, then $f_{F,K_3}(n)=\Theta_F(\sqrt{n \log n})$ for every triangle-free graph $F$. Note that this would also give a new proof of $R(3,t)=\Theta(t^2/\log t)$.

\begin{proposition}
	For every triangle-free graph $F$, if $c_F$ is sufficiently large, then the following holds. Let $m=c_F n^{1/2}(\log n)^{3/2}$. Assume that there is a $\{C_4,C_6\}$-free biregular bipartite graph with $n+m$ vertices and $\Omega((nm)^{2/3})$ edges. Then $f_{F,K_3}(n)\leq c_F\sqrt{n \log n}$.
\end{proposition}

\begin{remark}
	The biregularity assumption can be relaxed. Furthermore, any $C_6$-free graph can be made $C_4$-free by discarding at most half of its edges \cite{Gyori97}, so the same conclusion holds assuming the existence of a suitable $C_6$-free graph.
\end{remark}

\begin{proof}
	Let $H$ be a $\{C_4,C_6\}$-free biregular bipartite graph with $n+m$ vertices and $\Omega((nm)^{2/3})$ edges. Let $A$ be the part of size $n$ in $H$ and let $B$ be the part of size $m$. We define a graph $G$ on vertex set $A$ randomly as follows. For each $u\in B$, place a blow-up of $F$ randomly within $N_H(u)$ (each vertex in $N_H(u)$ is randomly allocated to one of the $|F|$-many sets corresponding to the vertices of $F$). The edge set of $G$ is the union of these random blow-ups of $F$.
	
	First, we claim that $G$ is triangle-free. Indeed, since $H$ is $\{C_4,C_6\}$-free, it is easy to check that any triangle must come from a single blow-up of $F$. But $F$ is triangle-free, so that is not possible.
	
	Now we will show that with positive probability, every set of $s=c_F\sqrt{n \log n}$ vertices in $A$ contains a copy of $F$ in $G$. Fix a set $S\subset A$ of size $s$. Since $H$ is biregular with $\Omega((nm)^{2/3})$ edges, every vertex in $A$ has degree $\Omega(m^{2/3}/n^{1/3})$ in $H$. Hence, $e_H(S,B)\geq \Omega(sm^{2/3}/n^{1/3})=\Omega(c_F^{5/3} n^{1/2} (\log n)^{3/2})=\Omega(c_F^{2/3} m)$.
	
	For a vertex $u\in B$, let $d_S(u)$ denote the number of neighbours of $u$ in $S$ in the graph $H$. Note that $\sum_{u\in B} d_S(u)=e_H(B,S)\geq \Omega(c_F^{2/3} m)$, so (if $c_F$ is large enough)
	$$\sum_{u\in B: d_S(u)\geq |F|} d_S(u)\geq \Omega(c_F^{2/3} m).$$
	
	But for any $u\in B$ with $d_S(u)\geq |F|$, the probability that $G[N(u)\cap S]$ contains no copy of $F$ is at most $\beta_F^{d_S(u)}$ for some constant $\beta_F<1$. Hence, the probability that $G[S]$ is $F$-free is at most $$\prod_{u\in B: d_S(u)\geq |F|} \beta_F^{d_S(u)}\leq \beta_F^{\Omega(c_F^{2/3} m)}.$$
	
	By taking a union bound over all choices of $S$, the probability that there is a set $S\subset A$ of size $s$ such that $G[S]$ is $F$-free is at most $$n^s\beta_F^{\Omega(c_F^{2/3} m)}= 2^{s \log n}\beta_F^{\Omega(c_F^{2/3} m)}=2^{m}\beta_F^{\Omega(c_F^{2/3} m)},$$
	which is less than $1$ provided that $c_F$ is sufficiently large. Hence, $f_{F,K_3}(n)\leq s=c_F \sqrt{n \log n}$.
\end{proof}

\section{Concluding remarks} \label{sec:concluding}

\subsection{Remark about improving some lower bounds in \cite{MV_general_graphs}}
We outline an argument which is implicit in \cite{Sudakov_DRC} and can be used to improve some of the lower bounds for
$f_{F,K_4}(n)$ proved in \cite{MV_general_graphs}. 
The improvement comes from the fact that the proof in \cite{MV_general_graphs} uses that a $K_4$-free graph with average degree $d$ has independence number at least $\sqrt{d}$; this follows by considering a vertex of degree at least $d$ and using the fact that a triangle-free graph with $m$ vertices has independence number at least $\sqrt{m}$. The following proposition gives a better bound in the relevant range of $d$.
\begin{proposition}[\cite{Sudakov_DRC}]\label{prop:DRC}
    Every $n$-vertex $K_4$-free graph with average degree $d \geq n^{2/3}$ contains an independent set of size $\Omega( \frac{d}{n^{1/3}})$.
\end{proposition}
\noindent
Note that the bound $\frac{d}{n^{1/3}}$ beats the bound $\sqrt{d}$ whenever $d \gg n^{2/3}$. When trying to prove a lower bound of the form $f_{F,K_4}(n) \geq n^{1/3 + \varepsilon}$, one can assume that the average degree $d$ of the host graph $G$ is at most $n^{2/3+2\varepsilon}$ (because otherwise $\alpha(G) \geq \sqrt{d} \geq n^{1/3+\varepsilon}$). This is part of the proof in \cite{MV_general_graphs}. By instead using Proposition \ref{prop:DRC}, one obtains the stronger $d \leq n^{2/3 + \varepsilon}$, which immediately leads to improved bounds (with the rest of the proof in \cite{MV_general_graphs} remaining the same). For example, one can improve the constant $\frac{1}{100}$ in the bound $f_{C_k,K_4}(n) \geq n^{\frac{1}{3} + \frac{1}{100k}}$ proved in \cite{MV_general_graphs}. Since we think that this may be useful in future works on this topic, we decided to include Proposition \ref{prop:DRC} and its proof. The proof uses the dependent random choice method \cite{FoxSudakov}.

\begin{proof}[Proof of Proposition \ref{prop:DRC}]
Suppose by contradiction that $\alpha := \alpha(G) \leq 0.1\frac{d}{n^{1/3}}$.
Let $K_{2,1,1}$ the diamond graph, i.e., the graph with vertices $x_1,x_2,y,z$ where $x_i,y,z$ is a triangle for $i=1,2$.
For each edge $yz$, the common neighbourhood of $y,z$ is an independent set (because $G$ is $K_4$-free), so it has size at most $\alpha$. Hence, the number of copies of $K_{2,1,1}$ is at most $e(G) \cdot \alpha^2 \leq dn\alpha^2$. Now sample two vertices $x_1,x_2$ uniformly at random and let $N$ be the common neighbourhood of $x_1,x_2$. By convexity, 
$\mathbb{E}[|N|] \geq \frac{d^2}{n}$. On the other hand, by the upper bound on the number of $K_{2,1,1}$'s, we have $\mathbb{E}[e(N)] \leq \frac{nd\alpha^2}{n^2} = \frac{d\alpha^2}{n}$. So by linearity of expectation,
$$
\mathbb{E}\left[ |N| - \frac{d^2}{2n} - 100e(N) \cdot \frac{\alpha n}{d^2} \right] \geq \frac{d^2}{2n} - 100\frac{d\alpha^2}{n} \cdot \frac{\alpha n}{d^2} = \frac{d^2}{2n} - 100\frac{\alpha^3}{d} \geq 0.
$$
So pick a choice of $N$ for which the above is nonnegative, meaning that $|N| \geq \frac{d^2}{2n} \geq \frac{d}{2n^{1/3}} \geq 5\alpha$ (using $d \geq n^{2/3}$) and $e(N) \leq \frac{1}{100\alpha} \cdot \frac{d^2}{n} \cdot |N| \leq \frac{|N|^2}{50\alpha}$. 
By Tur\'an's theorem, $G[N]$ contains an independent set of size at least 
$$
\frac{|N|^2}{2e(N) + |N|} \geq \min\left\{\frac{|N|}{3},\frac{|N|^2}{3e(N)} \right\} > \alpha,$$ a contradiction.
\end{proof}

\subsection{Open problems}
\begin{itemize}
    \item We proved that $n^{1/2 - O(1/\sqrt{t})} \leq f_{K_{t,t},K_4}(n) \leq n^{1/2 - \Omega(1/t)}$, with the lower bound coming from Theorem \ref{thm:Ktt vs K4} and the upper bound from Theorem \ref{thm:cliquefree vs clique precise}. It might be interesting to determine for this problem the correct order of magnitude of the error term in the exponent. 
    \item Another natural question is to estimate $f_{K_{2,t},K_4}(n)$. By an argument along the lines of the proof of Theorem \ref{thm:Ktt vs K4}, using also Proposition \ref{prop:DRC}, one can show that $f_{K_{2,t},K_4}(n) \geq n^{\frac{8}{21}-o_t(1)}$. We believe that it would be interesting to decide whether
$f_{K_{2,t},K_4}(n) \leq O(n^{1/2 - c})$ for some $c > 0$ which is independent of~$t$.
    \item 
    The construction of Mattheus and Verstra\"ete \cite{Mattheus_Verstraete} shows that the bound in Proposition \ref{prop:DRC} is tight (up to polylogarithmic terms) for $d=\Theta(n^{2/3})$. Here an interesting problem is to prove a tight bound for the size of the largest independent set one can guarantee in every $n$-vertex $K_4$-free graph with average degree
    $d = \Theta(n^{\alpha})$ for $\frac{2}{3} < \alpha < 1$. In particular, is there any $\alpha > \frac{2}{3}$ for which the bound given by Proposition \ref{prop:DRC} is tight? 
\end{itemize}

\bibliographystyle{abbrv}
\bibliography{library}

\end{document}